\documentclass[12pt,twoside]{amsart}
\usepackage{amsfonts}
\usepackage{amsmath}
\usepackage{amssymb}
\usepackage{mathrsfs}

\usepackage{color}
\usepackage{pst-eucl,pstricks-add}
\usepackage{pst-3dplot}
\usepackage{pst-solides3d}

\newtheorem{theo}{Theorem}
\newtheorem{cor}{Corollary}
\newtheorem{lem}{Lemma}
\newtheorem{prop}{Proposition}
\theoremstyle{definition}

\theoremstyle{remark}
\newtheorem{rem}{\bf Remark\/}

\numberwithin{equation}{section}

\def\1{{\mathchoice {\rm 1\mskip-4mu l} {\rm 1\mskip-4mu l}{\rm 1\mskip-4.5mu l} {\rm 1\mskip-5mu l}}}
\newcommand{\ds}{\displaystyle}

\title{Meromorphic Bergman spaces}

\author[N. Ghiloufi]{Noureddine Ghiloufi $^\star$}
\author[M. Zaway]{Mohamed Zaway $^{\star\star}$}

\keywords{Bergman spaces, Bergman Kernels, Hardy-Littlewood and Riesz-Fejer inequalities.\\
$^\star$Corresponding author,\\
E-mail addresses: nghiloufi@kfu.edu.sa, noureddine.ghiloufi@fsg.rnu.tn (N. Ghiloufi)
 and m\_zaway@su.edu.sa, mohamed\_zaway@yahoo.fr (M. Zaway)}
\begin{document}
\maketitle
$^\star$ {\footnotesize University of Gabes, Faculty of Sciences of Gabes, LR17ES11 Mathematics and Applications laboratory, 6072, Gabes, Tunisia.}\\
$^{\star\star}$ {\footnotesize Mathematics Department,
Faculty of Sciences and Humanities in Dawadmi, 
Shaqra University, 11911, Saudi Arabia.\\
Irescomath Laboratory, Gabes University, 6072 Zrig Gabes, Tunisia.}
\begin{abstract}
    In this paper we introduce new spaces of holomorphic functions on the pointed unit disc of $\mathbb C$  that generalize classical Bergman spaces. We prove some fundamental properties of these spaces and their dual spaces. We finish the paper by extending  Hardy-Littlewood and Fej\'er-Riesz inequalities to these spaces with an application on Toeplitz operators.
\end{abstract}

\section{Introduction and preliminary results}
Since the seventeenth of the last century the notion of Bergman spaces has known an increasing use in mathematics and essentially in complex analysis and geometry. The fundamental concept of this notion is the Bergman kernel. This kernel was computed firstly for the unit disc $\mathbb D$ in $\mathbb C$ and then it was  determined for any simply connected domain by the famous Riemann's theorem. However the determination  of the Bergman kernels of domains in $\mathbb C^n$ is more delicate and it is determined for some type of domains and still unknown up to our day in general. In this paper we generalize most properties of Bergman spaces of the unit disk by introducing new spaces of holomorphic functions on the pointed unit disc $\mathbb D^*$ that are square integrable with respect to a probability measure $d\mu_{\alpha,\beta}$ for some $\alpha,\beta>-1$. In fact the classical Bergman space is reduced to the case $\beta=0$ (See \cite{He-KO-Zh} for more details).  We call these new spaces meromorphic Bergman spaces; indeed any element of such a space is a meromorphic function which has $0$ as a pole of order controlled by the parameter $\beta$. The originality of our idea is that the Bergman kernels of these spaces may have zeros in the unit disk essentially when $\beta$ is not an integer. This problem will be discussed in a separate paper as a continuity of the present paper. For this reason we will concentrate here on the topological properties of these spaces and prove some well known inequalities.\\

Throughout this paper,  $\mathbb D(a,r)$ will be the disc of $\mathbb C$ with center  $a$ and radius $r>0$. In case $a=0$, we use  $\mathbb D(r)$ (resp. $\mathbb D$)  in stead of  $\mathbb D(0,r)$ (resp. $\mathbb D(0,1)$). We set $\mathbb S(r):=\partial \mathbb D(r)$ the circle and $\mathbb D^*:=\mathbb D\smallsetminus\{0\}$. For every $-1<\alpha,\beta<+\infty$, we consider the positive measure $\mu_{\alpha,\beta}$ on $\mathbb D$ defined by  $$d\mu_{\alpha,\beta}(z):=\frac{1}{\mathscr B(\alpha+1,\beta+1)}|z|^{2\beta}(1-|z|^2)^\alpha dA(z)$$ where $\mathscr B$ is the beta function defined by $$\mathscr B(s,t)=\int_0^1x^{s-1}(1-x)^{t-1}dx=\frac{\Gamma(s)\Gamma(t)}{\Gamma(s+t)},\quad \forall\; s,t>0$$ and $$dA(z)=\frac{1}{\pi}dx dy=\frac{1}{\pi}rdrd\theta,\quad z=x+iy=re^{i\theta}$$ the normalized area measure on $\mathbb D$.\\
The general aim of this paper is to study the properties of the Bergman type space $\mathcal A_{\alpha,\beta}^p(\mathbb D^*)$ defined for $0<p<+\infty$ as the set of holomorphic functions on $\mathbb D^*$ that belongs to the space
$$L^p(\mathbb D,d\mu_{\alpha,\beta})=\{f:\mathbb D\longrightarrow \mathbb C;\hbox{ measurable function such that } \|f\|_{\alpha,\beta,p}<+\infty\}$$
where $$\|f\|_{\alpha,\beta,p}^p:=\int_{\mathbb D}|f(z)|^pd\mu_{\alpha,\beta}(z).$$
When $1\leq p<+\infty$, the space $\left(L^p(\mathbb D,d\mu_{\alpha,\beta}),\|.\|_{\alpha,\beta,p}\right)$ is a Banach space; however for $0<p<1$, the space $L^p(\mathbb D,d\mu_{\alpha,\beta})$ is a complete metric space where the metric is given by $d(f,g)=\|f-g\|_{\alpha,\beta,p}^p$. The following proposition will be useful in the hole of the paper.
\begin{prop}\label{p1}
    For every $0<r<1$ and $0<\varepsilon<1$  there exists  $c_\varepsilon(r)=c_{\varepsilon,\alpha,\beta}(r)>0$ such that for any $0<p<+\infty$ and $f\in \mathcal A_{\alpha,\beta}^p(\mathbb D^*)$ we have
    $$|f(z)|^p\leq \frac{ \mathscr B(\alpha+1,\beta+1) }{c_\varepsilon(r)}\|f\|_{\alpha,\beta,p}^p,\qquad \forall z\in \mathbb S(r).$$
    One can choose $c_\varepsilon(r)=r_\varepsilon^2\mathfrak{a}_\varepsilon(r)\mathfrak{b}_\varepsilon(r)$ with $r_\varepsilon=\varepsilon \min(r,1-r)$,
    $$\mathfrak{a}_\varepsilon(r):=\left\{
    \begin{array}{lcl}
      \displaystyle\left[1-(r+r_\varepsilon)^2\right]^\alpha& if& \alpha\geq 0\\
      \\
       \displaystyle\left[1-(r-r_\varepsilon)^2\right]^\alpha& if& -1<\alpha<0
    \end{array}\right.$$
    and
    $$\mathfrak{b}_\varepsilon(r):=\left\{
    \begin{array}{lcl}
      \displaystyle(r-r_\varepsilon)^{2\beta}& if& \beta\geq 0\\
      \\
       \displaystyle(r+r_\varepsilon)^{2\beta}& if& -1<\beta<0
    \end{array}\right.$$
\end{prop}
\begin{proof}
    Let $0<r<1$,  $0<\varepsilon<1$ and $0<p<+\infty$ be fixed reals. Let $f\in \mathcal A_{\alpha,\beta}^p(\mathbb D^*)$ and $z\in \mathbb S(r)$.  We set $r_\varepsilon=\varepsilon \min(r,1-r)$. It is easy to see that $\overline{\mathbb D}(z,r_\varepsilon)\subset \mathbb D^*$; so thanks to the subharmonicity of $|f|^p$, we obtain
    \begin{equation}\label{eq 1.1}
    \begin{array}{lcl}
        \displaystyle|f(z)|^p&\leq& \displaystyle\frac{1}{r_\varepsilon^2}\int_{\mathbb D(z,r_\varepsilon)}|f(w)|^pdA(w)\\
        &\leq & \ds\frac{\mathscr B(\alpha+1,\beta+1)}{r_\varepsilon^2}\int_{\mathbb D(z,r_\varepsilon)} \frac{|f(w)|^p}{|w|^{2\beta}(1-|w|^2)^\alpha} d\mu_{\alpha,\beta}(w)\\
      \end{array}
    \end{equation}
    If $w\in\mathbb D(z,r_\varepsilon)$ then $r-r_\varepsilon\leq |w|\leq r+r_\varepsilon$. Thus we obtain
    $$|w|^{2\beta}\geq \mathfrak{b}_\varepsilon(r):=\left\{
    \begin{array}{lcl}
      \displaystyle(r-r_\varepsilon)^{2\beta}& if& \beta\geq 0\\
      \\
       \displaystyle(r+r_\varepsilon)^{2\beta}& if& -1<\beta<0
    \end{array}\right.$$
    and
    $$(1-|w|^2)^\alpha\geq\mathfrak{a}_\varepsilon(r):=\left\{
    \begin{array}{lcl}
      \displaystyle\left[1-(r+r_\varepsilon)^2\right]^\alpha& if& \alpha\geq 0\\
      \\
       \displaystyle\left[1-(r-r_\varepsilon)^2\right]^\alpha& if& -1<\alpha<0
    \end{array}\right.$$
    It follows that Inequality (\ref{eq 1.1}) gives
    $$\begin{array}{lcl}
        \displaystyle|f(z)|^p&\leq& \ds\frac{\mathscr B(\alpha+1,\beta+1)}{r_\varepsilon^2\mathfrak{a}_\varepsilon(r)\mathfrak{b}_\varepsilon(r)}\int_{\mathbb D(z,r_\varepsilon)} |f(w)|^p d\mu_{\alpha,\beta}(w)\\
        &\leq& \ds\frac{\mathscr B(\alpha+1,\beta+1)}{r_\varepsilon^2\mathfrak{a}_\varepsilon(r)\mathfrak{b}_\varepsilon(r)}\|f\|_{\alpha,\beta,p}^p.
      \end{array}$$
\end{proof}
Using the previous proof, one can improve the previous proposition as follows
\begin{rem}
    For any $n\in\mathbb N$ and $0<r<1$, there exists $c=c(n,r,\alpha,\beta)>0$ such that for every $f\in\mathcal A_{\alpha,\beta}^p(\mathbb D^*)$ we have
    $$|f^{(n)}(z)|^p\leq c\|f\|_{\alpha,\beta,p}^p,\qquad \forall z\in \mathbb S(r).$$
\end{rem}
As a first consequence of Proposition \ref{p1}, we have
\begin{cor}
    For every $-1<\alpha,\beta<+\infty$ and $0<p<+\infty$, the space $\mathcal A_{\alpha,\beta}^p(\mathbb D^*)$ is closed in $L^p(\mathbb D,\mu_{\alpha,\beta})$ and for any $z\in\mathbb D^*$, the linear form $\delta_z:\mathcal A_{\alpha,\beta}^p(\mathbb D^*)\longrightarrow \mathbb C$ defined by $\delta_z(f)=f(z)$ is bounded on $\mathcal A_{\alpha,\beta}^p(\mathbb D^*)$.
\end{cor}
\begin{proof}
    As $L^p(\mathbb D,d\mu_{\alpha,\beta})$ is complete, it suffices to consider a sequence $(f_n)_n\subset \mathcal A_{\alpha,\beta}^p(\mathbb D^*)$ that converges to $f\in L^p(\mathbb D,d\mu_{\alpha,\beta})$ and to prove that $f\in  \mathcal A_{\alpha,\beta}^p(\mathbb D^*)$. Thanks to Proposition \ref{p1}, the sequence $(f_n)_n$ converges uniformly to $f$ on every compact subset of $\mathbb D^*$. Hence the function $f$ is holomorphic on $\mathbb D^*$ and we conclude that  $f\in  \mathcal A_{\alpha,\beta}^p(\mathbb D^*)$.\\
    For the second statement, one can see that $\delta_z$ is a linear functional well defined on $\mathcal A_{\alpha,\beta}^p(\mathbb D^*)$. For the continuity of $\delta_z$, thanks to Proposition \ref{p1}, for every $z\in\mathbb D^*$, there exists $c>0$ such that for every $f\in\mathcal A_{\alpha,\beta}^p(\mathbb D^*)$ we have $|\delta_z(f)|=|f(z)|\leq c\|f\|_{\alpha,\beta,p}$. Thus the linear functional $\delta_z$ is continuous on $\mathcal A_{\alpha,\beta}^p(\mathbb D^*)$.
\end{proof}
In the following we give some immediate properties:\\
\begin{itemize}
  \item if $f\in\mathcal A_{\alpha,\beta}^p(\mathbb D^*)$ then $0$ can't be an essential singularity for $f$, hence either $0$ is removable for $f$ (so $f$ is holomorphic on $\mathbb D$) or $0$ is a pole for $f$ with order $\nu_f=\nu_f(0)$ that satisfies
\begin{equation}\label{eq 1.2}
\nu_f\leq m_{p,\beta}=\left\{\begin{array}{lcl}
                        \ds\left\lfloor\frac{2(\beta+1)}{p}\right\rfloor & if & \ds\frac{2(\beta+1)}{p}\not\in\mathbb N\\
                        \ds\frac{2(\beta+1)}{p}-1& if & \ds\frac{2(\beta+1)}{p}\in\mathbb N
                      \end{array}\right.
\end{equation}

    where $\lfloor .\rfloor$ is the integer part.
    \item If we set $\widetilde{f}(z)=z^{\nu_f}f(z)$ then $\widetilde{f}$ is a holomorphic function  on $\mathbb D$ and $f\in\mathcal A_{\alpha,\beta}^p(\mathbb D^*)$ if and only if $\widetilde{f}\in\mathcal A_{\alpha,\beta-\frac{p\nu_f}{2}}^p(\mathbb D^*)$ and
        $$\|f\|_{\alpha,\beta,p}=\left(\frac{\mathscr B(\alpha+1,\beta-\frac{p\nu_f}{2}+1)}{\mathscr B(\alpha+1,\beta+1)}\right)^{\frac1p}\|\widetilde{f}\|_{\alpha,\beta-\frac{p\nu_f}{2},p}.$$
        Using the two previous properties, if we replace $f$ by $z^{m_{p,\beta}}f$ in the proof of Proposition \ref{p1}, we can obtain a more sharp estimate in Proposition \ref{p1}.
  \item If $-1<\beta<\beta'$ and $-1<\alpha<\alpha'$ then $\mathcal A_{\alpha,\beta}^p(\mathbb D^*)\subseteq\mathcal A_{\alpha',\beta'}^p(\mathbb D^*)$ and the canonical injection is continuous. This is a consequence of the fact that we have
      $$\mathscr B(\alpha'+1,\beta'+1)\|f\|_{\alpha',\beta',p}\leq \mathscr B(\alpha+1,\beta+1)\|f\|_{\alpha,\beta,p}$$ for every $f\in\mathcal A_{\alpha,\beta}^p(\mathbb D^*)$.
\end{itemize}
Claim that if we set $\mathbb D_\zeta:=\mathbb D\smallsetminus\{\zeta\}$ for any $\zeta\in\mathbb D$, then  all results on $\mathcal A_{\alpha,\beta}^p(\mathbb D^*)$ can be extended to the space $\mathcal A_{\alpha,\beta}^p(\mathbb D_\zeta)$ of holomorphic functions on $\mathbb D_\zeta$ that are $p-$integrable with respect to the positive measure $|z-\zeta|^{2\beta}(1-|z|^2)^\alpha dA(z)$. Indeed $h\in\mathcal A_{\alpha,\beta}^p(\mathbb D_\zeta)$ if and only if $h\circ\varphi_\zeta\in\mathcal A_{\alpha,\beta}^p(\mathbb D^*)$ where $\varphi_\zeta(z)=\frac{\zeta-z}{1-\overline{\zeta}z}$.
\section{Meromorphic Bergman Kernels}

In the case $p=2$ we have $\mathcal A_{\alpha,\beta}^2(\mathbb D^*)$ is a Hilbert space and $\mathcal A_{\alpha,\beta}^2(\mathbb D^*)=\mathcal A_{\alpha,m}^2(\mathbb D^*)$ for every $\beta\in]m-1,m]$ with $m\in\mathbb N$. If we set
\begin{equation}\label{eq 2.1}
e_n(z)=\sqrt{\frac{\mathscr B(\alpha+1,\beta+1)}{\mathscr B(\alpha+1,n+\beta+1)}}\ z^n
\end{equation}
for every $n\geq -m$, then the sequence $(e_n)_{n\geq -m}$ is a Hilbert basis of $\mathcal A_{\alpha,\beta}^2(\mathbb D^*)$. Furthermore, if  $f,g\in\mathcal A_{\alpha,\beta}^2(\mathbb D^*)$ with $$f(z)=\sum_{n= -m}^{+\infty}a_nz^n,\quad g(z)=\sum_{n= -m}^{+\infty}b_nz^n$$  then $$\langle f,g\rangle_{\alpha,\beta}=\sum_{n= -m}^{+\infty}a_n\overline{b}_n\frac{\mathscr B(\alpha+1,n+\beta+1)}{\mathscr B(\alpha+1,\beta+1)}$$ where $\langle .,.\rangle_{\alpha,\beta}$ is the inner product in $\mathcal A_{\alpha,\beta}^2(\mathbb D^*)$ inherited from $L^2(\mathbb D,d\mu_{\alpha,\beta})$.
\begin{lem}\label{l1}
Let $-1<\alpha<+\infty$ and  $m\in\mathbb N$. Then the reproducing (Bergman) kernel $\mathbb{K}_{\alpha,m}$  of $\mathcal A_{\alpha,m}^2(\mathbb D^*)$ is given by
$$\ds \mathbb{K}_{\alpha,m}(w,z)=\frac{(\alpha+1)\mathscr B(\alpha+1,m+1)}{(w\overline{z})^m(1-w\overline{z})^{2+\alpha}}.$$
\end{lem}
\begin{proof}
    The sequence $(e_n)_{n\geq -m}$ given by (\ref{eq 2.1}) is a Hilbert basis of $\mathcal A_{\alpha,\beta}^2(\mathbb D^*)$, hence  the reproducing kernel of $\mathcal A_{\alpha,\beta}^2(\mathbb D^*)$ is given by
   $$
    \begin{array}{lcl}
         \mathbb{K}_{\alpha,\beta}(w,z)&=&\ds\sum_{n=-m}^{+\infty}e_n(w)\overline{e_n(z)}=\sum_{n=-m}^{+\infty}\frac{\mathscr B(\alpha+1,\beta+1)}{\mathscr B(\alpha+1,n+\beta+1)}w^n\overline{z}^n\\
          &=&\ds\frac{1}{(w\overline{z})^m}\sum_{n=0}^{+\infty}\frac{\mathscr B(\alpha+1,\beta+1)}{\mathscr B(\alpha+1,n+\beta-m+1)}(w\overline{z})^n
      \end{array}$$
      The computation of this kernel in the general case is more complicated. However, in our case for $\beta=m$, we obtain
    $$
    \begin{array}{lcl}
         \mathbb{K}_{\alpha,m}(w,z)&=&\ds\frac{1}{(w\overline{z})^m}\sum_{n=0}^{+\infty}\frac{\mathscr B(\alpha+1,m+1)}{\mathscr B(\alpha+1,n+1)}(w\overline{z})^n\\
          &=&\ds \frac{(\alpha+1)\mathscr B(\alpha+1,m+1)}{(w\overline{z})^m(1-w\overline{z})^{2+\alpha}}.
      \end{array}$$
\end{proof}
Here we give some fundamental properties of the Bergman kernel as  consequences of Lemma \ref{l1}.
\begin{cor}
    Let $-1<\alpha<+\infty$ and $m\in\mathbb N$. Let $\mathbb{P}_{\alpha,m}$ be the orthogonal projection from $L^2(\mathbb D,d\mu_{\alpha,m})$ onto $\mathcal A_{\alpha,m}^2(\mathbb D^*)$. Then for every $f\in L^2(\mathbb D,d\mu_{\alpha,m})$ we have
    $$\mathbb{P}_{\alpha,m}f(z)=(\alpha+1)\mathscr B(\alpha+1,m+1)\int_{\mathbb D}\frac{f(w)} {(z\overline{w})^m(1-z\overline{w})^{2+\alpha}}d\mu_{\alpha,m}(w).$$
\end{cor}
\begin{proof}
    This is a simple consequence of  Lemma \ref{l1} and the fact that for every  $f\in L^2(\mathbb D,d\mu_{\alpha,m})$ we have
    $\mathbb{P}_{\alpha,m}f(z)=\langle f,\mathbb{K}_{\alpha,m}(.,z)\rangle_{\alpha,m}.$
\end{proof}
Using the density of $\mathcal A_{\alpha,m}^2(\mathbb D^*)$ in $\mathcal A_{\alpha,m}^1(\mathbb D^*)$, one can prove the following corollary:
\begin{cor} Let $-1<\alpha<+\infty$ and $m\in\mathbb N$. Then for every $f\in \mathcal A_{\alpha,m}^1(\mathbb D^*)$ we have
    $$f(z)=(\alpha+1)\mathscr B(\alpha+1,m+1)\int_{\mathbb D}\frac{f(w)}{(z\overline{w})^m(1-z\overline{w})^{2+\alpha}}d\mu_{\alpha,m}(w).$$
\end{cor}
The following result is well known in general, its proof is based essentially on the fact that $\mathbb{K}_{\alpha,\beta}(z,z)\neq 0$ for every
$z\in\mathbb D^*$.
\begin{prop}
Let $-1<\alpha,\beta<+\infty$ and $\mathbb{K}_{\alpha,\beta}$ be the reproducing (Bergman) kernel of $\mathcal A_{\alpha,\beta}^2(\mathbb D^*)$.  Then for every $z\in\mathbb D^*$ we have $\mathbb{K}_{\alpha,\beta}(z,z)>0$ and satisfies
\begin{equation}\label{eq 2.2}
\begin{array}{lcl}
    \ds \mathbb{K}_{\alpha,\beta}(z,z)&=&\ds \sup\left\{|f(z)|^2;\ f\in\mathcal A_{\alpha,\beta}^2(\mathbb D^*),\ \|f\|_{\alpha,\beta,2}\leq1\right\}\\
    &=&\ds\sup\left\{\frac{1}{\|f\|_{\alpha,\beta,2}};\ f\in\mathcal A_{\alpha,\beta}^2(\mathbb D^*),\ f(z)=1\right\}
  \end{array}
\end{equation}
In particular, the norm of the Dirac form $\delta_z$ on $\mathcal A_{\alpha,\beta}^2(\mathbb D^*)$ is given by $$\|\delta_z\|=\|\mathbb{K}_{\alpha,\beta}(.,z)\|_{\alpha,\beta,2}=\sqrt{\mathbb{K}_{\alpha,\beta}(z,z)}.$$
\end{prop}
One can find the proof of the first equality in Krantz book \cite{Kr}, however the second one is due to Kim \cite{Ki}. For the completeness of our paper we give the proof.
\begin{proof}

Thanks to the proof of Lemma \ref{l1}, we have  $\mathbb{K}_{\alpha,\beta}(z,z)>0$ for every $z\in\mathbb D^*$. To prove the first equality in (\ref{eq 2.2}), we fix $z\in\mathbb D^*$ and we consider
$$\mathscr Q(z):=\sup\left\{|f(z)|^2;\ f\in\mathcal A_{\alpha,\beta}^2(\mathbb D^*),\ \|f\|_{\alpha,\beta,2}\leq1\right\}.$$
Let $f\in\mathcal A_{\alpha,\beta}^2(\mathbb D^*)$ such that $\|f\|_{\alpha,\beta,2}\leq1$. Then thanks to the Cauchy-Schwarz inequality,
$$|f(z)|^2=|\langle f,\mathbb{K}_{\alpha,\beta}(.,z)\rangle_{\alpha,\beta}|^2\leq \|f\|_{\alpha,\beta,2}^2\|\mathbb{K}_{\alpha,\beta}(.,z)\|_{\alpha,\beta,2}^2\leq \mathbb{K}_{\alpha,\beta}(z,z).$$
It follows that $\mathscr Q(z)\leq \mathbb{K}_{\alpha,\beta}(z,z)$. Conversely, we set
$$g(\xi)=\frac{\mathbb{K}_{\alpha,\beta}(\xi,z)}{\sqrt{\mathbb{K}_{\alpha,\beta}(z,z)}},\quad \xi\in\mathbb D^*.$$
Hence we have $g\in\mathcal A_{\alpha,\beta}^2(\mathbb D^*)$, $\|g\|_{\alpha,\beta,2}=1$ and $|g(z)|^2=\mathbb{K}_{\alpha,\beta}(z,z)$ and the converse inequality $\mathscr Q(z)\geq \mathbb{K}_{\alpha,\beta}(z,z)$ is proved.\\
Now to prove the second equality in (\ref{eq 2.2}), we let
$$\mathscr M(z):=\inf\left\{\|f\|_{\alpha,\beta,2};\ f\in\mathcal A_{\alpha,\beta}^2(\mathbb D^*),\ f(z)=1\right\}$$
If we set $$h(\xi)=\frac{\mathbb{K}_{\alpha,\beta}(\xi,z)}{\mathbb{K}_{\alpha,\beta}(z,z)},\quad \xi\in\mathbb D^*$$
then  $h(z)=1$ and $h\in\mathcal A_{\alpha,\beta}^2(\mathbb D^*)$. Indeed,
$$\begin{array}{lcl}
    \|h\|_{\alpha,\beta,2}^2&=&\ds\int_{\mathbb D}|h(\xi)|^2d\mu_{\alpha,\beta}(\xi)=\int_{\mathbb D} \frac{\mathbb{K}_{\alpha,\beta}(\xi,z)}{\mathbb{K}_{\alpha,\beta}(z,z)} \frac{\mathbb{K}_{\alpha,\beta}(z,\xi)}{\mathbb{K}_{\alpha,\beta}(z,z)}d\mu_{\alpha,\beta}(\xi) \\ &=&\ds\frac{1}{\mathbb{K}_{\alpha,\beta}(z,z)^2}\mathbb{K}_{\alpha,\beta}(z,z)=\frac{1}{\mathbb{K}_{\alpha,\beta}(z,z)}.
  \end{array}
$$
It follows that $$\mathscr M(z)\leq \frac{1}{\mathbb{K}_{\alpha,\beta}(z,z)}.$$
Conversely, for every $f\in\mathcal A_{\alpha,\beta}^2(\mathbb D^*)$ such that $f(z)=1$ we have
$$|f(\zeta)|^2=|\langle f,\mathbb{K}_{\alpha,\beta}(.,\zeta)\rangle_{\alpha,\beta}|^2\leq \|f\|_{\alpha,\beta,2}^2\mathbb{K}_{\alpha,\beta}(\zeta,\zeta).$$
Thus we obtain
$$\frac{|f(\zeta)|^2}{\mathbb{K}_{\alpha,\beta}(\zeta,\zeta)}\leq \|f\|_{\alpha,\beta,2}^2,\quad \forall\; \zeta\in\mathbb D^*.$$
In particular, for $\zeta=z$,
$$\frac{1}{\mathbb{K}_{\alpha,\beta}(z,z)}\leq \|f\|_{\alpha,\beta,2}^2.$$
We conclude that
$$\frac{1}{\mathbb{K}_{\alpha,\beta}(z,z)}\leq \mathscr M(z).$$
\end{proof}

\section{Duality  of meromorphic Bergman spaces}

The aim of this part is to prove that the dual of $\mathcal A_{\alpha,\beta}^p(\mathbb D^*)$ is related to $\mathcal A_{\alpha,\beta}^q(\mathbb D^*)$ with $\frac1p+\frac1q=1$. This will be a consequence of the main result (Theorem \ref{th1}). But to prove the main result we need  the following lemma:
\begin{lem}\label{l2}
    For every $-1<\sigma, \gamma<+\infty$ we set
    $$I_\omega(z)=\int_{\mathbb D}\frac{(1-|w|^2)^{\sigma}|w|^{2\gamma}}{|1-z\overline{w}|^{2+\sigma+\omega}}dA(w).$$
    Then $I_\omega$ is continuous on $\mathbb D$ and
    $$I_\omega(z)\sim \left\{
    \begin{array}{lcl}
      1&if& \omega<0\\
      \ds\log\frac{1}{1-|z|^2}&if& \omega=0\\
      \ds\frac{1}{(1-|z|^2)^\omega}&if& \omega>0
    \end{array}\right.$$
    when $|z|\longrightarrow 1^-$ where $\varphi\sim \psi$ means that there exist $0<c_1<c_2$ such that we have $c_1\varphi(z)\leq \psi(z)\leq c_2\varphi(z)$.
\end{lem}
\begin{proof}
The proof is similar to \cite[Theo.1.7]{He-KO-Zh}.
\end{proof}
\begin{theo}\label{th1}
For every $-1<\alpha, a, b<+\infty$ and $m\in\mathbb N$, we consider the two integral operators $T$ and $S$ defined by
$$\begin{array}{lcl}
     Tf(z)&=&\ds\frac{1}{z^m}\int_{\mathbb D}\frac{f(w)(1-|w|^2)^{\alpha-a} w^m}{|w|^{2b}(1-z\overline{w})^{2+\alpha}}d\mu_{a,b}(w)\\
     Sf(z)&=&\ds\frac{1}{|z|^m}\int_{\mathbb D}\frac{f(w)(1-|w|^2)^{\alpha-a} |w|^{m-2b}}{|1-z\overline{w}|^{2+\alpha}}d\mu_{a,b}(w).
  \end{array}
$$
Then for every $1\leq p<+\infty$, the following assertions are equivalent:
\begin{enumerate}
  \item $T$ is bounded on $L^p(\mathbb D,d\mu_{a,b})$
  \item $S$ is bounded on $L^p(\mathbb D,d\mu_{a,b})$
  \item $ p(\alpha+1)>a+1$  and
        $\left\{\begin{array}{lcl}
                   \ds m-2<2b\leq m& if& p=1\\
                   \ds mp-2<2b<mp-2+2p& if& p>1
                \end{array}\right.$
\end{enumerate}
\end{theo}
\begin{proof}
$"(2)\Longrightarrow (1)"$ is obvious.\\
$"(1)\Longrightarrow (2)"$ can be deduced using the transformation
$$\varOmega_z f(w)=\frac{(1-z\overline{w})^{2+\alpha}|w|^m}{|1-z\overline{w}|^{2+\alpha}w^m}f(w).$$
$"(2)\Longrightarrow (3)"$ Now assume that $S$ is bounded on $L^p(\mathbb D,d\mu_{a,b})$. If we apply $S$ to  $f_N(z)=(1-|z|^2)^N$ for $N$ large enough, we obtain
$$\|Sf_N\|_{a,b,p}^p=\ds \int_{\mathbb D}(1-|z|^2)^a|z|^{2b-mp}\left(\int_{\mathbb D}\frac{(1-|w|^2)^{\alpha+N} |w|^{m}}{|1-z\overline{w}|^{2+\alpha}}dA(w)\right)^pdA(z)$$
is finite. Thanks to Lemma \ref{l2}, we obtain $b>\frac{mp}{2}-1$.\\
To prove the other inequalities,  we need $S^\star$ the adjoint operator of $S$ with respect to the inner product $\langle.,.\rangle_{a,b}$. It is given by
$$\begin{array}{lcl}
     S^\star g(w)&=&\ds(1-|w|^2)^{\alpha-a} |w|^{m-2b}\int_{\mathbb D}\frac{g(z)}{|z|^m|1-z\overline{w}|^{2+\alpha}}d\mu_{a,b}(z)\\
     &=&\ds(1-|w|^2)^{\alpha-a} |w|^{m-2b}\int_{\mathbb D}\frac{g(z)|z|^{2b-m}(1-|z|^2)^a}{|1-z\overline{w}|^{2+\alpha}}dA(z).
  \end{array}$$

We distinguish two cases:
\begin{enumerate}
  \item First case $p=1$: $S$ is bounded on  $L^1(\mathbb D,d\mu_{a,b})$ gives  $S^\star$ is bounded on $L^\infty(\mathbb D,d\mu_{a,b})$. By applying $S^\star$ on the constant function $g\equiv 1$, we obtain
      $$\sup_{w\in\mathbb D^*}(1-|w|^2)^{\alpha-a} |w|^{m-2b}\int_{\mathbb D}\frac{|z|^{2b-m}(1-|z|^2)^a}{|1-z\overline{w}|^{2+\alpha}}dA(z)<+\infty.$$
      Thanks to Lemma \ref{l2}, we get $m-2b\geq0$ and $\alpha-a>0$. The desired inequalities are proved.
  \item Second case $p>1$: Let $q>1$ such that $\frac{1}{p}+\frac{1}{q}=1$. Again  by applying $S^\star$ on the function $f_N$ for $N$ large enough, we obtain
    $$\begin{array}{l}
        \|S^\star f_N\|_{a,b,q}^q=\\
        =\ds\int_{\mathbb D}(1-|w|^2)^{a+q(\alpha-a)} |w|^{2b+(m-2b)q}\left(\int_{\mathbb D}\frac{|z|^{2b-m}(1-|z|^2)^{a+N}}{|1-z\overline{w}|^{2+\alpha}}dA(z)\right)^qdA(w)
      \end{array}
    $$
    is finite and hence all inequalities $$\ds\frac{mp}{2}-1<b<\frac{mp}{2}-1+p$$ and $ p(\alpha+1)>a+1$ hold.
\end{enumerate}
$"(3)\Longrightarrow (2)"$  We start by the case $p=1$. We assume that $m-2<2b\leq m$ and $\alpha>a$. Using Lemma \ref{l2}, one can prove easily the boundedness of $S$ on $L^1(\mathbb D,d\mu_{a,b})$.\\
Now for $p>1$, to prove the boundedness of $S$ on $L^p(\mathbb D,d\mu_{a,b})$ we will use the Schur's test. We set
$$h(z)=\frac1{|z|^t(1-|z|^2)^s},\quad \hbox{and}\quad  \kappa(z,w)= \frac{(1-|w|^2)^{\alpha-a} |w|^{m-2b}}{|z|^m|1-z\overline{w}|^{2+\alpha}}.$$
Thanks to Lemma \ref{l2}, if
\begin{equation}\label{eq 3.1}
    \frac mq\leq t<\frac{m+2}q,\quad 0<s<\frac{\alpha+1}q
\end{equation}
then
$$\begin{array}{lcl}
    \ds\int_{\mathbb D} \kappa(z,w)h(w)^qd\mu_{a,b}(w)&=&\ds\frac1{|z|^m}\int_{\mathbb D}\frac{(1-|w|^2)^{\alpha-sq} |w|^{m-tq}}{|1-z\overline{w}|^{2+\alpha}}dA(w)\\
    &\leq&\ds \frac{c_1}{|z|^m(1-|z|^2)^{sq}}=c_1.|z|^{tq-m}h(z)^q\leq c_1.h(z)^q
  \end{array}
$$
for some positive constant $c_1>0$. \\
Similarly, if
\begin{equation}\label{eq 3.2}
    \frac{2b-m}p\leq t<\frac{2b-m+2}p,\quad \frac{a-\alpha}{p}<s<\frac{a+1}p
\end{equation}
then
$$\begin{array}{l}
    \ds\int_{\mathbb D} \kappa(z,w)h(z)^pd\mu_{a,b}(z)\\
    =\ds(1-|w|^2)^{\alpha-a} |w|^{m-2b}\int_{\mathbb D} \frac{|z|^{2b-m-tp}(1-|z|^2)^{a-sp}}{|1-z\overline{w}|^{2+\alpha}}dA(z)\\
    \leq\ds c_2.\frac{|w|^{m-2b}}{(1-|w|^2)^{sp}}=c_2.|w|^{m-2b+tp}h(w)^p\leq c_2.h(w)^p
  \end{array}
$$
with $c_2>0$. Thanks to the hypothesis given in $(3)$, we have
$$\left]\frac{m}{q},\frac{m+2}q\right[\cap\left]\frac{2b-m}p,\frac{2b-m+2}p\right[\neq\emptyset,\quad \left]0,\frac{\alpha+1}q\right[\cap \left]\frac{a-\alpha}{p},\frac{a+1}p\right[\neq\emptyset.$$
This proves the existence of $t$ and $s$ satisfying (\ref{eq 3.1}) and (\ref{eq 3.2}). Thanks to Schur's test, $S$ is bounded on $L^p(\mathbb D,d\mu_{a,b})$.
\end{proof}

\begin{theo}
    For every $1<p<+\infty$ and $-1<a,b<+\infty$, the topological dual of $\mathcal A_{a,b}^p(\mathbb D^*)$ is the space $\mathcal A_{a,b}^q(\mathbb D^*)$ under the integral pairing
    $$\langle f,g\rangle_{a,b}=\int_{\mathbb D}f(z)\overline{g(z)}d\mu_{a,b}(z),\quad \forall\; f\in \mathcal A_{a,b}^p(\mathbb D^*),\ g\in \mathcal A_{a,b}^q(\mathbb D^*)$$
    where $q$ is the conjugate exponent of $p$.
\end{theo}
\begin{proof}
Thanks to H\"older inequality, every function $g\in \mathcal A_{a,b}^q(\mathbb D^*)$ defines a bounded linear form on $\mathcal A_{a,b}^p(\mathbb D^*)$ via the above integral pairing. Conversely, Let $G$ be a bounded linear form on $\mathcal A_{a,b}^p(\mathbb D^*)$. Then thanks to Hahn Banach extension theorem, one can extend $G$ to a bounded linear form on $L^p(\mathbb D, d\mu_{a,b})$ (still denoted by $G$) with the same norm. By duality, there exists  $\psi\in L^q(\mathbb D, d\mu_{a,b})$ such that
$$G(f)=\langle f,\psi\rangle_{a,b},\quad \forall\; f\in \mathcal A_{a,b}^p(\mathbb D^*).$$
Claim that if $m=m_{p,b}$ given in (\ref{eq 1.2}), then thanks to Theorem \ref{th1}, $\mathbb P_{a,m}$ maps continuously  $L^p(\mathbb D, d\mu_{a,b})$ onto $\mathcal A_{a,b}^p(\mathbb D^*)$ and $\mathbb P_{a,m}f=f$ for every $f\in\mathcal A_{a,b}^p(\mathbb D^*)$. It follows that
$$G(f)=\langle f,\psi\rangle_{a,b}=\langle \mathbb P_{a,m}f,\psi\rangle_{a,b}=\langle f,\mathbb P_{a,m}^\star\psi\rangle_{a,b},\quad \forall\; f\in \mathcal A_{a,b}^p(\mathbb D^*).$$
If we  set $g=\mathbb P_{a,m}^\star\psi$ then $g\in \mathcal A_{a,b}^q(\mathbb D^*)$ and $G(f)=\langle f,g\rangle$ for every $f\in \mathcal A_{a,b}^p(\mathbb D^*).$
\end{proof}

\section{Inequalities on $\mathcal A_{\alpha,\beta}^p(\mathbb D^*)$}
The aim here is to extend the two famous Hardy-Littlewood and Fej\'er-Riesz inequalities to our new spaces, these inequalities were proved firstly on Hardy spaces, then on Bergman spaces with some applications. In our case we give only one application on Toeplitz operators. To reach this aim, for  a holomorphic function $f$ on $\mathbb D^*$ and $0<r<1$, we consider the main value on the circle:
    $$\begin{array}{lcl}
        M_p(r,f)&:=&\ds\left(\frac{1}{2\pi}\int_0^{2\pi}|f(re^{i\theta})|^pd\theta\right)^{\frac1p}\\
         M_\infty(r,f)&:=&\ds\sup_{\theta\in[0,2\pi]}|f(re^{i\theta})|.
      \end{array}$$
      We set $$\mathscr J(r)=\mathscr J_{\alpha,\beta,p}(r):=\frac{2r^{pm_{p,\beta}}}{\mathscr B(\alpha+1,\beta+1)}\int_r^1t^{2\beta-pm_{p,\beta}+1}(1-t^2)^\alpha dt.$$
\subsection{Hardy-Littlewood inequality}
To prove the Hardy-Littlewood inequality on $\mathcal A_{\alpha,\beta}^p(\mathbb D^*)$, we need to prove firstly the following lemma:
    \begin{lem}\label{l3}
        For every $p>1$ and $f\in\mathcal A_{\alpha,\beta}^p(\mathbb D^*)$, we have
        $$M_p(r,f)\leq  \frac{\|f\|_{\alpha,\beta,p}}{\mathscr J(r)^{\frac1p}}.$$
        In particular we have
        $$M_p(r,f)\leq  \frac{\kappa_1\|f\|_{\alpha,\beta,p}}{r^{\max(\frac{2\beta}p,\nu_f)}(1-r^2)^{\frac{\alpha+1}p}}$$
        where $\kappa_1=\left((\alpha+1)\mathscr B(\alpha+1,\beta+1)\right)^{\frac1p}.$
    \end{lem}
    \begin{proof}
        Let $f\in\mathcal A_{\alpha,\beta}^p(\mathbb D^*)$ and $0<r<1$. We set $F(z)=z^mf(z)$ with $m=m_{p,\beta}$. As $F$ is holomorphic on $\mathbb D$, then we obtain
        $$\begin{array}{lcl}
        \ds \|f\|_{\alpha,\beta,p}^p&=&\ds \frac{1}{\pi \mathscr B(\alpha+1,\beta+1)}\int_0^1\int_0^{2\pi}|f(te^{i\theta})|^p(1-t^2)^{\alpha}t^{2\beta+1} dtd\theta\\
        &=&\ds \frac{2}{\mathscr B(\alpha+1,\beta+1)}\int_0^1M_p^p(t,f)(1-t^2)^{\alpha}t^{2\beta+1}dt\\
        &=&\ds \frac{2}{\mathscr B(\alpha+1,\beta+1)}\int_0^1M_p^p(t,F)(1-t^2)^{\alpha}t^{2\beta-pm+1}dt\\
        &\geq&\ds \frac{2}{\mathscr B(\alpha+1,\beta+1)}M_p^p(r,F)\int_r^1(1-t^2)^{\alpha}t^{2\beta-pm+1}dt\\
        &=&\ds M_p^p(r,f)\mathscr J(r)
      \end{array}$$
      and the first inequality is proved. The particular case can be deduced from the following inequality:
      $$\begin{array}{lcl}
           \mathscr J(r)&=&\ds\frac{r^{pm}}{\mathscr B(\alpha+1,\beta+1)}\int_{r^2}^1(1-t)^{\alpha}t^{\beta-pm/2}dt\\
           & & \\
           &\geq &\ds \left\{\begin{array}{lcl}
                               \ds\frac{r^{2\beta}(1-r^2)^{\alpha+1}}{(\alpha+1)\mathscr B(\alpha+1,\beta+1)}& if & 2\beta \geq pm\\
                               & & \\
                               \ds\frac{r^{pm}(1-r^2)^{\alpha+1}}{(\alpha+1)\mathscr B(\alpha+1,\beta+1)}& if & 2\beta<pm
                             \end{array}\right.\\
           & & \\
           & \geq& \ds \frac{r^{\max(2\beta,pm)}(1-r^2)^{\alpha+1}}{(\alpha+1)\mathscr B(\alpha+1,\beta+1)}.
        \end{array}
      $$
    \end{proof}
    Now we can prove the Hardy-Littlewood inequality on meromorphic Bergman spaces:
    \begin{theo}
        For every $1<p\leq \tau\leq \infty$, there exists a positive constant $\kappa$ such that for every $f\in\mathcal A_{\alpha,\beta}^p(\mathbb D^*)$ we have
        $$M_\tau(r,f)\leq \frac{\kappa\|f\|_{\alpha,\beta,p}}{r^{\max(\frac{2\beta}p,m)}(1-r^2)^{\frac{\alpha+2}p-\frac1\tau}}
        $$
        where $m=m_{p,\beta}$.
    \end{theo}
    The Hardy-Littlewood inequality is proved in \cite{So} for classical Bergman spaces ($\beta=0$).
    \begin{proof}
        The case $\tau=p$ is simply the previous lemma. Let we start by the case $\tau=\infty$.  Let $f\in\mathcal A_{\alpha,\beta}^p(\mathbb D^*)$ and $0<r<1$. We set $F(z)=z^{m}f(z)$. Again as $F$ is holomorphic on $\mathbb D$, then thanks to the Cauchy Formula, we have
        $$F(re^{i\theta})=\frac{s}{2\pi}\int_0^{2\pi}\frac{s^{m}e^{im t}f(s e^{it})}{s e^{it}-re^{i\theta}}e^{it}dt$$
        where $s=\frac{1+r}2$. Applying H\"{o}lder's inequality  $(\frac1p+\frac1{q}=1)$ and Lemma \ref{l3}, we obtain
        $$\begin{array}{lcl}
             r^{m}|f(re^{i\theta})|&\leq & \ds\left(\frac1{2\pi}\int_0^{2\pi}s^{pm}|f(s e^{it})|^pdt\right)^{\frac1p} \left(\frac1{2\pi}\int_0^{2\pi}\frac{s^{q}}{|s e^{it}-re^{i\theta}|^{q}}dt\right)^{\frac1{q}}\\
             &\leq&\ds s^{m}M_p(s,f)\left(\frac{\kappa_2}{\left(\frac{1-r}{1+r}\right)^{q-1}}\right)^{\frac1{q}}\\
             &\leq&\ds s^{m} \frac{\kappa_1\|f\|_{\alpha,\beta,p}}{s^{\max(\frac{2\beta}p,m)}(1-s^2)^{\frac{\alpha+1}p}} \frac{\kappa_3}{(1-r^2)^{1-\frac1{q}}}\\
             &\leq&\ds\frac{\kappa_4\|f\|_{\alpha,\beta,p}}{r^{\max(\frac{2\beta}p-m,0)}(1-r^2)^{\frac{\alpha+2}p}}
          \end{array}
        $$
        It follows that
        $$M_\infty(r,f)\leq \frac{\kappa_4\|f\|_{\alpha,\beta,p}}{r^{\max(\frac{2\beta}p,m)}(1-r^2)^{\frac{\alpha+2}p}}.$$
        Let now $p<\tau<\infty$. We have
        $$\begin{array}{lcl}
             M_\tau(r,f)&= &\ds \left(\frac1{2\pi}\int_0^{2\pi}|f(re^{it})|^p|f(re^{it})|^{\tau-p}dt\right)^{\frac1\tau}\\
             &\leq&\ds \ds M_\infty^{1-\frac p\tau}(r,f)M_p^{\frac p\tau}(r,f)\\
             &\leq&\ds \left( \frac{\kappa_4\|f\|_{\alpha,\beta,p}}{r^{\max(\frac{2\beta}p,m)}(1-r^2)^{\frac{\alpha+2}p}}\right)^{1-\frac p\tau} \left(\frac{\kappa_1\|f\|_{\alpha,\beta,p}}{r^{\max(\frac{2\beta}p,m)}(1-r^2)^{\frac{\alpha+1}p}}\right)^{\frac p\tau}\\
             &=& \ds \frac{\kappa\|f\|_{\alpha,\beta,p}}{r^{\max(\frac{2\beta}p,m)}(1-r^2)^{\frac{\alpha+2}p-\frac1\tau}}
          \end{array}
        $$
    \end{proof}

    \subsection{Fej\'er-Riesz inequality}
    The aim here is to prove a generalization of the following lemma to meromorphic Bergman spaces.
    \begin{lem}\label{l4}(See \cite{Du})
        Let $g$  be a holomorphic function in the Hardy space $H^p(\mathbb D)$. Then for any $\xi\in\mathbb C$ with $|\xi|=1$, we have
        $$\int_{-1}^1|g(t\xi)|^pdt\leq \frac12 \|g\|_{H^p}^p:=\frac12\int_0^{2\pi}|g(e^{i\theta})|^pd\theta.$$
    \end{lem}
    \begin{theo}[Fej\'er-Riesz inequality]
        For every $f\in\mathcal A_{\alpha,\beta}^p(\mathbb D^*)$ and $\xi\in\mathbb C$ with $|\xi|=1$, we have
        $$\int_{-1}^1|f(t\xi)|^p\mathscr J(|t|)dt\leq  \pi\|f\|_{\alpha,\beta,p}^p.$$
    \end{theo}
    Claim that if $\beta=0$ then we find the Zhu result \cite{Zh}.
    \begin{proof}
        Let $f\in\mathcal A_{\alpha,\beta}^p(\mathbb D^*)$ and $F(z)=z^{m}f(z)$ where $m=m_{p,\beta}$. If we set  $F_r(z)=F(rz)$ for $0<r<1$, Then $F_r\in H^p(\mathbb D)$ and thanks to Lemma \ref{l4}, for every  $\xi\in\mathbb C,\ |\xi|=1$,
        $$\int_{-1}^1|F_r(t\xi)|^pdt\leq \frac12\int_0^{2\pi}|F_r(e^{i\theta})|^pd\theta.$$
        That is
        \begin{equation}\label{eq 4.1}
        \int_{-1}^1|F(rt\xi)|^pdt\leq \frac12\int_0^{2\pi}|F(re^{i\theta})|^pd\theta.
        \end{equation}
        Thanks to Inequality (\ref{eq 4.1})  and Fubini theorem, we have
        $$\begin{array}{l}
             \|f\|_{\alpha,\beta,p}^p=\\
             =\ds \frac{1}{\pi \mathscr B(\alpha+1,\beta+1)}\int_0^1\int_0^{2\pi}|f(re^{i\theta})|^pr^{2\beta+1}(1-r^2)^{\alpha} drd\theta\\
             =\ds \frac{1}{\pi \mathscr B(\alpha+1,\beta+1)}\int_0^1\left(\int_0^{2\pi}|F(re^{i\theta})|^pd\theta\right)r^{2\beta-pm+1}(1-r^2)^{\alpha} dr\\
             \geq\ds \frac{2}{\pi \mathscr B(\alpha+1,\beta+1)}\int_0^1\left(\int_{-1}^1|F(rt\xi)|^pdt\right)r^{2\beta-pm+1}(1-r^2)^{\alpha} dr\\
             =\ds \frac{2}{\pi \mathscr B(\alpha+1,\beta+1)}\int_0^1\left(\int_{-r}^r|F(s\xi)|^pds\right)r^{2\beta-pm}(1-r^2)^{\alpha} dr\\
             =\ds \frac{2}{\pi \mathscr B(\alpha+1,\beta+1)}\int_{-1}^1|F(s\xi)|^p\left(\int_{|s|}^1r^{2\beta-pm}(1-r^2)^{\alpha} dr\right)ds\\
             =\ds \frac{2}{\pi \mathscr B(\alpha+1,\beta+1)}\int_{-1}^1|f(s\xi)|^p\left(|s|^{pm}\int_{|s|}^1r^{2\beta-pm}(1-r^2)^{\alpha} dr\right)ds\\
             \geq\ds \frac{1}{\pi}\int_{-1}^1|f(s\xi)|^p\mathscr J(|s|)ds.
          \end{array}
        $$
    \end{proof}
As an application of the Fej\'er-Riesz inequality on the Toeplitz operators, we have the following result:
\begin{theo}
For every $\xi\in\mathbb D^*$, if we consider the Toeplitz operator $\mathcal T$ defined by
$$\mathcal Tf(z)=\int_{-1}^1f(\xi x)\mathbb K_{\alpha,\beta}(z,\xi x)\mathscr J_{\alpha,\beta,2}(|x|)dx.$$
Then $\mathcal T$ is a positive bounded linear operator on $\mathcal A_{\alpha,\beta}(\mathbb D^*)$.
\end{theo}
When $\beta=0$, this result is due to Andreev \cite{An} proved in a restricted case.
\begin{proof}
Thanks to Fubini theorem, for every $f\in \mathcal A_{\alpha,\beta}(\mathbb D^*)$ one has
$$\begin{array}{lcl}
    \langle \mathcal Tf,f\rangle_{\alpha,\beta}&=&\ds \int_{\mathbb D}\mathcal Tf(z)\overline{f(z)}d\mu_{\alpha,\beta}(z)\\
    &=&\ds \int_{\mathbb D}\left(\int_{-1}^1f(\xi x)\mathbb K_{\alpha,\beta}(z,\xi x)\mathscr J_{\alpha,\beta,2}(|x|)dx\right) \overline{f(z)}d\mu_{\alpha,\beta}(z)\\
    &=&\ds \int_{-1}^1f(\xi x)\overline{\int_{\mathbb D}\mathbb K_{\alpha,\beta}(\xi x,z)f(z)d\mu_{\alpha,\beta}(z)}\mathscr J_{\alpha,\beta,2}(|x|)dx\\
    &=&\ds \int_{-1}^1f(\xi x)\overline{f(\xi x)}\mathscr J_{\alpha,\beta,2}(|x|)dx\\
    &\leq & \pi\|f\|_{\alpha,\beta,2}^2.
  \end{array}
$$
The last inequality is the Fej\'er-Riesz one in the particular case $p=2$.\\
This proves that the operator $\mathcal T$ is positive and thus it is self-adjoint and bounded with norm $\|\mathcal T\|\leq\pi.$ Indeed,
$$\|\mathcal T\|=\sup\{|\langle \mathcal Tf,f\rangle_{\alpha,\beta}|;\ \|f\|_{\alpha,\beta,2}=1\}\leq \pi.$$
\end{proof}


\begin{thebibliography}
{X-XX1}
\bibitem{An}\textbf{V. V. Andreev}, Fej\'er-Riesz type inequalities for Bergman spaces,  Rend. Circ. Mat. Palermo 61 (2012) 385-392.
\bibitem{Du}\textbf{P. L. Duren}, Theory of $H^p$ spaces, Academic press (1970).
\bibitem{He-KO-Zh}\textbf{H. Hedenmalm, B. Korenblum and K. Zhu}, Theory of Bergman spaces, Graduate texts in Mathematics, 199 (2000).
\bibitem{Ki}\textbf{H. Kim}, On the localization of the minimum integral related to the weighted Bergman kernel and its application, C. R. Acad. Sci. Paris, Ser. I 355 (2017) 420-425.
\bibitem{Kr}\textbf{S.G. Krantz}, Geometric analysis of the Bergman kernel and metric, Graduate text in Mathematics 268, (2013).
\bibitem{So}\textbf{P. Sobolewski}, Inequalities on Bergman spaces, Ann. Univ. Marie Curie-Sk{\l}odowska Lublin Polonia Vol. LXI (2007) 137-143.
\bibitem{Zh}\textbf{K. Zhu}, Translating Inequalities between Hardy and Bergman spaces, Mathematical Assoc. Amer. Monthly 111 (2004) 520-525.
\end{thebibliography}
\end{document}